\documentclass[draft]{elsarticle}

\usepackage{graphicx,courier,amssymb,amsmath,times,amsthm}
\usepackage{epsfig}
\usepackage{bm,graphicx,amssymb,psfrag}
\usepackage{amsfonts,eucal}
\usepackage[latin1]{inputenc}

\usepackage{algorithmic}
\usepackage{algorithm}

\usepackage{hyperref}

\usepackage[usenames]{color}

\usepackage[nolist]{acronym}

\newcommand{\W}{\mathcal{TW}}
\newcommand{\X}{\mathcal{G}}
\def\nmin{s}

\def\Nmin{s}

\def\Hypergeometric1F1{${}_{1}F_{1}$}

\def\sgn{\text{sgn}}

\def\aappx{\mathsf{\alpha}}
\def\kappx{\mathsf{k}}
\def\tappx{\mathsf{\delta}}

\def\x{\theta}
\def\X{\Theta}
\def\betareg{{\mathcal{B}}}
\def\I{{\mathcal{E}}}

\def\p{\mathsf{p}}
\def\n{\mathsf{n}}
\def\m{\mathsf{m}}

\def\gammareg{{{P}}}

\newtheorem{theorem}{Theorem}

\journal{JMVA}

\begin{document}

\begin{acronym}
\scriptsize
\acro{AcR}{autocorrelation receiver}
\acro{ACF}{autocorrelation function}
\acro{ADC}{analog-to-digital converter}
\acro{AWGN}{additive white Gaussian noise}
\acro{BCH}{Bose Chaudhuri Hocquenghem}
\acro{BEP}{bit error probability}
\acro{BFC}{block fading channel}
\acro{BPAM}{binary pulse amplitude modulation}
\acro{BPPM}{binary pulse position modulation}
\acro{BPSK}{binary phase shift keying}
\acro{BPZF}{bandpass zonal filter}
\acro{CD}{cooperative diversity}
\acro{CDF}{cumulative distribution function}
\acro{CCDF}{complementary cumulative distribution function}
\acro{CDMA}{code division multiple access}
\acro{c.d.f.}{cumulative distribution function}
\acro{ch.f.}{characteristic function}
\acro{CIR}{channel impulse response}
\acro{CR}{cognitive radio}
\acro{CSI}{channel state information}
\acro{DAA}{detect and avoid}
\acro{DAB}{digital audio broadcasting}
\acro{DS}{direct sequence}
\acro{DS-SS}{direct-sequence spread-spectrum}
\acro{DTR}{differential transmitted-reference}
\acro{DVB-T}{digital video broadcasting\,--\,terrestrial}
\acro{DVB-H}{digital video broadcasting\,--\,handheld}
\acro{ECC}{European Community Commission}
\acro{ELP}{equivalent low-pass}
\acro{FCC}{Federal Communications Commission}
\acro{FEC}{forward error correction}
\acro{FFT}{fast Fourier transform}
\acro{FH}{frequency-hopping}
\acro{FH-SS}{frequency-hopping spread-spectrum}
\acro{GA}{Gaussian approximation}
\acro{GPS}{Global Positioning System}
\acro{HAP}{high altitude platform}
\acro{i.i.d.}{independent, identically distributed}
\acro{IFFT}{inverse fast Fourier transform}
\acro{IR}{impulse radio}
\acro{ISI}{intersymbol interference}
\acro{LEO}{low earth orbit}
\acro{LOS}{line-of-sight}
\acro{BSC}{binary symmetric channel}
\acro{MB}{multiband}
\acro{MC}{multicarrier}
\acro{MF}{matched filter}
\acro{m.g.f.}{moment generating function}
\acro{MI}{mutual information}
\acro{MIMO}{multiple-input multiple-output}
\acro{MISO}{multiple-input single-output}
\acro{MRC}{maximal ratio combiner}
\acro{MMSE}{minimum mean-square error}
\acro{M-QAM}{$M$-ary quadrature amplitude modulation}
\acro{M-PSK}{$M$-ary phase shift keying}
\acro{MUI}{multi-user interference}
\acro{NB}{narrowband}
\acro{NBI}{narrowband interference}
\acro{NLOS}{non-line-of-sight}
\acro{NTIA}{National Telecommunications and Information Administration}
\acro{OC}{optimum combining}
\acro{OFDM}{orthogonal frequency-division multiplexing}
\acro{p.d.f.}{probability distribution function}
\acro{PAM}{pulse amplitude modulation}
\acro{PAR}{peak-to-average ratio}
\acro{PDP}{power dispersion profile}
\acro{p.m.f.}{probability mass function}
\acro{PN}{pseudo-noise}
\acro{PPM}{pulse position modulation}
\acro{PRake}{Partial Rake}
\acro{PSD}{power spectral density}
\acro{PSK}{phase shift keying}
\acro{QAM}{quadrature amplitude modulation}
\acro{QPSK}{quadrature phase shift keying}
\acro{r.v.}{random variable}
\acro{R.V.}{random vector}
\acro{SEP}{symbol error probability}
\acro{SIMO}{single-input multiple-output}
\acro{SIR}{signal-to-interference ratio}
\acro{SISO}{single-input single-output}
\acro{SINR}{signal-to-interference plus noise ratio}
\acro{SNR}{signal-to-noise ratio}
\acro{SS}{spread spectrum}
\acro{TH}{time-hopping}
\acro{ToA}{time-of-arrival}
\acro{TR}{transmitted-reference}
\acro{UAV}{unmanned aerial vehicle}
\acro{UWB}{ultrawide band}
\acro{UWBcap}[UWB]{Ultrawide band}
\acro{WLAN}{wireless local area network}
\acro{WMAN}{wireless metropolitan area network}
\acro{WPAN}{wireless personal area network}
\acro{WSN}{wireless sensor network}
\acro{WSS}{wide-sense stationary}

\acro{SW}{sync word}
\acro{FS}{frame synchronization}
\acro{BSC}{binary symmetric channels}
\acro{LRT}{likelihood ratio test}
\acro{GLRT}{generalized likelihood ratio test}
\acro{LLRT}{log-likelihood ratio test}
\acro{$P_{EM}$}{probability of emulation, or false alarm}
\acro{$P_{MD}$}{probability of missed detection}
\acro{ROC}{receiver operating characteristic}
\acro{AUB}{asymptotic union bound}
\acro{RDL}{"random data limit"}
\acro{PSEP}{pairwise synchronization error probability}

\acro{SCM}{sample covariance matrix}

\acro{PCA}{principal component analysis}

\end{acronym}

\begin{frontmatter}
\title{Distribution of the largest root of a matrix for Roy's test in multivariate analysis of variance 
}
\author[label1]{Marco~Chiani}
\address[label1]{ 
    DEI, University of Bologna\\
    V.le Risorgimento 2, 40136 Bologna, ITALY\\
    marco.chiani@unibo.it
}

\begin{abstract}
Let ${\bf X, Y} $ denote two independent real Gaussian $\p \times \m$ and $\p \times \n$ matrices with $\m, \n \geq \p$, each constituted by zero mean \acl{i.i.d.} columns with common covariance.
 The Roy's largest root criterion, used in multivariate analysis of variance (MANOVA), is based on the statistic of the largest eigenvalue, $\X_1$, of ${\bf{(A+B)}}^{-1} \bf{B}$, where ${\bf A =X X}^T$ and ${\bf B =Y Y}^T$ are independent central Wishart matrices. 
We  derive a new expression and efficient recursive formulas for the exact distribution of $\X_1$. The expression can be easily calculated even for large parameters, eliminating the need of pre-calculated tables for the application of the Roy's test. 
\end{abstract}


\begin{keyword}
Roy's test \sep Random Matrices \sep multivariate analysis of variance (MANOVA) \sep characteristic roots \sep largest eigenvalue \sep Tracy-Widom distribution \sep Wishart matrices.
\end{keyword}
\end{frontmatter}

\section{Introduction}
\label{sec:intro}

The joint distribution  of $s$ non-null eigenvalues of a multivariate real beta matrix in the null case can be written in the form \cite[page 112]{Mui:B82}, \cite[page 331]{And:03}, 
\begin{equation}\label{eq:jpdfuncorrnodet}
f(\x_{1}, \ldots, \x_{\Nmin}) = C(s,m,n) \,
     \prod_{i=1}^{\Nmin} \x_{i}^{m} (1-\x_{i})^n
    \cdot    \prod_{i<j}^{\Nmin}
    \left(\x_{i}-\x_{j}\right) 
\end{equation}
where $1 > \x_{1} \geq \x_{2} \cdots \geq \x_{\Nmin} > 0$, and $C(s,m,n)$ is a normalizing constant given by
\begin{equation}
C=C(s,m,n) = \pi ^{s/2} \prod _{i=1}^s  \frac{ \Gamma \left(\frac{i+2 m+2 n+s+2}{2} \right)}{ \Gamma
   \left(\frac{i}{2}\right) \Gamma \left(\frac{i+2 m+1}{2} \right) \Gamma \left(\frac{i+2 n+1}{2} \right)}\, .
\label{eq:C}
\end{equation}
This distribution arises in multivariate analysis of variance (MANOVA) and, with the notation introduced above, is the distribution of the eigenvalues of ${\bf (A+B)^{-1}B}$ with parameters
\begin{equation}\label{eq:par}
s=\p, \qquad m=(\n-\p-1)/2 , \qquad  n=(\m-\p-1)/2 \,.
\end{equation}
The marginal distribution of the largest eigenvalue, $\X_1$, is of basic importance in testing hypotheses and constructing confidence regions in multivariate analysis of variance (MANOVA) according to the Roy's largest root criterion 
(see  e.g. \cite[page 333]{And:03} and references therein), and is generally considered difficult to compute. For this reason, extensive studies have produced tables of upper percentage points for few specific (small) values of $s$ and some combinations of $m, n$ (see e.g. \cite{Hec:60, Pil:67}, \cite[Table B.4]{And:03}
). The  most efficient numerical algorithm to compute the \ac{CDF} of $\X_1$ is provided in \cite{But:11} based on \cite{GupRic:85}, but for nonintegers $m$ or $n$ it requires infinite series expansions which can results in computational time of several hours. 
 Approximations based on the Tracy-Widom distribution are discussed in \cite{Joh:09}.

In this paper we derive a simple expression for the exact \ac{CDF} of $\X_1$  for arbitrary $s, n, m$, and an iterative algorithm for its fast evaluation. The algorithm needs only the incomplete beta function, and does not rely on numerical integration or series expansion. For instance, all results in \cite[Table B.4]{And:03} or in \cite[Table 1]{Joh:09} can be easily computed instantaneously\footnote{On a current desktop computer in less than $0.1$ seconds.}; even for the most challenging cases  analyzed in literature 
 \cite[Table 1]{But:11}, which with previous methods required hours of computational time,  no more than a second is needed. Finally, we discuss some approximations based on the Tracy-Widom distribution and its approximation \cite{Joh:09,Chi:J14}.

We remark that we study the case of real matrices: the complex analogous of our problem, i.e., the case where ${\bf X, Y} $ are independent complex Gaussian, is much easier and has been solved in \cite{Kha:64}. 

\smallskip
Throughout the paper we indicate with $\Gamma(.)$ the gamma function, with $B(a,b)$  the  beta function, with $\betareg\left(x;a,b\right)=\int_{0}^{x} t^{a-1} (1-t)^{b-1} dt$ the  incomplete (lower) beta function \cite[Ch. 6]{AbrSte:B70}, and with $|\cdot|$ the determinant.


\section{Exact distribution of the largest eigenvalue for multivariate beta matrices in the null case}
\label{sec:exact}

\mbox{ }

The following is the main result of the paper.

\begin{theorem}
\label{th:cdfwishartreal} 
The \ac{CDF} of the largest eigenvalue $\X_1$ for a multivariate beta matrix in the null case is:
\begin{equation}
\label{eq:cdfwishartreal}
F_{\X_1}(\x_1)=\Pr\left\{\X_1 \leq \x_1\right\}=C \,   \sqrt{\left|{\bf A}(\x_1)\right|} \,.
\end{equation}
When $\nmin$ is even, the elements of the $\nmin \times \nmin$ skew-symmetric matrix ${\bf A}(\x_1)$  are:
\begin{equation}
\label{eq:aij}
a_{i,j}(\x_1)=
\I\left(\x_1; m+j,m+i\right)-\I\left(\x_1; m+i,m+j\right) \qquad i,j=1,\ldots,\nmin
\end{equation}
where
\begin{equation}
\label{eq:I}
\I(x;a,b)\triangleq \int_{0}^{x} t^{a-1} (1-t)^n \, \betareg\left(t; b,n+1 \right)dt .
\end{equation}
When $\nmin$ is odd, the elements of the $(\nmin+1) \times (\nmin+1)$ skew-symmetric matrix ${\bf A}(x_1)$ are as in \eqref{eq:aij}, with the additional elements
\begin{eqnarray}  \label{eq:aijodd}
  a_{i,\nmin+1}(\x_1)&=&\betareg\left(\x_1; m+i,n+1 \right)  \qquad i=1, \ldots, \nmin  \\
\label{eq:aijodd1} a_{\nmin+1,j}(\x_1)&=&-a_{j,\nmin+1}(\x_1)  \qquad\qquad j=1, \ldots, \nmin 
\\   
a_{\nmin+1,\nmin+1}(\x_1)&=&0 
\end{eqnarray}
\noindent 
Note that $a_{i,j}(\x_1)=-a_{j,i}(\x_1)$ and $a_{i,i}(\x_1)=0$.

\medskip
\noindent 
Moreover, the elements $a_{i,j}(\x_1)$ can be computed iteratively, starting from the beta function, without numerical integration or series expansion.
\end{theorem}

\begin{proof}
The proof is based on the approach introduced in \cite{Chi:J14} for Wishart and GOE matrices.

Denoting $\xi(x)=x^m (1-x)^n$, ${\bf x} = \left[x_{1}, x_{2}, \ldots, x_{\Nmin} \right]$, and with ${\bf V}({\bf x})=\left\{x_j^{i-1}\right\}$ the Vandermonde matrix, 
  we have for the eigenvalues in ascending order
\begin{equation}
\label{eq:freversed}
 f (x_{\Nmin}, \ldots, x_{1})=C  \prod_{i<j} (x_j-x_i) \prod_{i=1}^{\Nmin} \xi(x_i)= C  \left|{\bf V}({\bf x})\right| \prod_{i=1}^{\Nmin} \xi\left(x_i\right) 
 \end{equation}
 where now  $0 < x_1 \leq \cdots \leq x_{\Nmin} <1$. 

The \ac{CDF} of the largest eigenvalue is  then
\begin{eqnarray}
F_{\X_1}(\x_1)&=&\underset{{0 \leq x_{1} < \ldots < x_{\Nmin} \leq \x_1} }{\int\ldots\int} f(x_{\Nmin}, \ldots, x_{1}) d{\bf x} \\
&=&C \underset{{0 \leq x_1 < \ldots < x_{\Nmin} \leq \x_1}}{\int\ldots\int}   \left|{\bf V}({\bf x})\right|
     \prod_{i=1}^{\Nmin} \xi\left(x_i\right)  d{\bf x} \,. \label{eq:th1proof}
\end{eqnarray}

To evaluate this integral we recall that for a generic $s \times s$ matrix ${\bf \Phi}({\bf x})$ with elements $\left\{\Phi_i(x_j)\right\}$ the following identity holds \cite{Deb:55}  
\begin{equation}
\label{eq:debru}
\underset{{a \leq x_{1} < \ldots < x_{\Nmin} \leq  b} }{\int\ldots\int}  \left|{\bf \Phi}({\bf x})\right|  d{\bf x} = \text{Pf}\left({\bf A}\right)
\end{equation}
where $\text{Pf}\left({\bf A}\right)=\sqrt{\left|{\bf A}\right|}$ is the Pfaffian, and the skew-symmetric matrix $\bf A$ is $s \times s$ for $s$ even, and $(s+1) \times (s+1)$ for $s$ odd, with 
\begin{equation}
\label{eq:aijdebru}
a_{i,j}=\int_a^b \int_a^b \sgn(y-x) \Phi_i(x) \Phi_j(y) dx dy \qquad i,j = 1, \ldots, s .
\end{equation}
For $s$ odd the additional elements are $a_{i,s+1}=-a_{s+1,i}=\int_a^b \Phi_i(x) dx$, $i=1, \ldots, s$, and $a_{s+1,s+1}=0$.

Using \eqref{eq:debru} in \eqref{eq:th1proof} with $a=0, b=\x_1, \Phi_i(x)=x^{i-1} \xi(x)=x^{i-1} x^m (1-x)^n$ with some simple manipulations gives Theorem \ref{th:cdfwishartreal}.

%
%
\smallskip
To avoid the numerical integration in \eqref{eq:I}, we first observe that the incomplete beta functions can be computed iteratively by the relation
\begin{equation}
\label{eq:betarecursive}
\betareg(x; a+1, b)=\frac{a}{a+b} \betareg(x; a, b)-  \frac{x^a (1-x)^{b}}{a+b}.
\end{equation}
This relation is obtained for example with integration by parts of $\int t^a(1-t)^{b-1} dt=\int f(t) g'(t) dt$ with $f(t)=t^a-t^{a+1}$, $g'(t)=(1-t)^{b-2}$.

Then, from \eqref{eq:I} and \eqref{eq:betarecursive} the following identities can be easily verified:
\begin{eqnarray}
\I(x;a,a)&=& \frac{1}{2} \betareg\left(x; a, n+1 \right)^2  \label{eq:rec1} \\
\displaystyle \I(x;a,b+1)&=&b \, \frac{\I(x; a,b)}{b+n+1}-\frac{\betareg(x; a+b,2n+2)}{b+n+1}    \label{eq:rec2} \\
\I(x;b,a)&=& \betareg\left(x;a,n+1 \right) \betareg\left(x;b,n+1 \right) -\I(x;a,b) \,. \label{eq:rec3}
\end{eqnarray}
Therefore, the elements of the matrix in  \eqref{eq:cdfwishartreal} can be evaluated iteratively without any numerical integration.
\end{proof}
Theorem \ref{th:cdfwishartreal} can thus be used for an efficient computation of the exact \ac{CDF} of the largest eigenvalue for Roy's test.  
In fact, using \eqref{eq:rec1}, \eqref{eq:rec2}, and \eqref{eq:rec3}, the \ac{CDF} in \eqref{eq:cdfwishartreal} can be simply evaluated, without numerical integrations or series expansion, by Algorithm \ref{alg} reported below. 

\begin{algorithm}[!ht]
\renewcommand{\algorithmicrequire}{\textbf{Input:}}
\renewcommand{\algorithmicensure}{\textbf{Output:}}
\begin{algorithmic}[0]
\REQUIRE $\nmin, m, n, \x_1$
\renewcommand{\algorithmicrequire}{\textbf{Needed:}}
\REQUIRE
 \STATE Incomplete beta function $\betareg\left(x; a, b \right)$ 
 
\ENSURE $F_{\X_1}(\x_1)=\Pr\left\{\X_1 \leq \x_1\right\}$
 \STATE ${\bf A}={\bf 0}$
 
   \FOR{$i = 1 \to \nmin$}
  	\STATE $ \displaystyle b_{i} = \betareg(\x_1; m+i,n+1)^2/2$ 
  	\FOR {$j = i \to \nmin-1$}
  		\STATE $ \displaystyle b_{j+1}=\frac{m+j}{m+j+n+1}  b_{j}-\frac{\betareg(\x_1; 2m+i+j,2n+2)}{m+j+n+1} $ 
		\STATE $ \displaystyle a_{i,j+1}=\betareg\left(\x_1;m+i,n+1 \right) \betareg\left(\x_1;m+j+1,n+1 \right)-2 b_{j+1}$
  	\ENDFOR
  \ENDFOR
  \IF{$\nmin$ is odd}
     \STATE Append one column to ${\bf A}$ with $a_{i,s+1}=\betareg\left(\x_1;m+i,n+1 \right), \, i=1, \ldots, s$
     \STATE Append one zero row to ${\bf A}$ 
  \ENDIF
  \STATE ${\bf A}={\bf A}-{\bf A}^T$
   \RETURN $F_{\X_1}(\x_1)= C(s,m,n) \, \sqrt{|{\bf A}|}$
\end{algorithmic}
\caption{Algorithm: \ac{CDF} of the largest eigenvalue for Roy's test}
\label{alg}
\end{algorithm}
%

\noindent
Implementing directly the algorithm in Mathematica on a personal computer, for each value $x_1$ we obtain the exact \ac{CDF} in \eqref{eq:cdfwishartreal} in less than $0.1$ seconds for all tables in \cite{Pil:67}, \cite[Table B.4]{And:03} and \cite[Table 1]{Joh:09}. 
For $s=54, m=-1/2, n=45/2$, which corresponds to the most challenging case analyzed in literature \cite[Table 1, last row]{But:11}, the exact \ac{CDF} is computed with the new expression in less than a second. For comparison, the approach based on series expansions of the hypergeometric function of a matrix argument requires hours. For larger parameters for which no other methods are available in literature, like for example $s=200, m=-1/2, n=299/2$, the computational time is less than fifteen seconds on a common  personal computer. Therefore, the upper percentage points for Roy's test can be the evaluated exactly and almost instantaneously for parameters of interest in applied statistics.

\section{Approximations}
In \cite{Joh:09} it is shown that the logit of $\X_1$ approaches the Tracy-Widom law for large $s$. More precisely, it is shown that, for $m \geq -1/2, n \geq 0$ and large $s$, the following approximation holds 
\begin{equation}\label{eq:logit}
\frac{\log\left(\X_1/(1-\X_1)\right)-\mu}{\sigma} \overset{{\mathcal{D}}}{\approx} \W_1
\end{equation}
where $\W_{1}$ denotes a \acl{r.v.} with Tracy-Widom distribution of order $1$ \cite{TraWid:94,TraWid:96,Joha:00,Joh:01,TraWid:09}, and the values of $\mu, \sigma$ are given, using the parameters in \eqref{eq:par}, by \cite{Joh:08,Joh:09}
\begin{eqnarray}
\mu &=& 2 \log \tan \left(\frac{\gamma +\phi }{2}\right) \\
\sigma^3&=&  \frac{16}{(\m+\n-1)^2} \frac{1}{\sin^2(\gamma +\phi) \sin\gamma \sin\phi }\\
\gamma&=&\arccos\left(\frac{\m+\n-2\p}{\m+\n-1}\right) \\ 
\phi&=&\arccos\left(\frac{\m-\n}{\m+\n-1}\right)  \,.
\end{eqnarray}
Moreover, in \cite{Chi:J14} it is shown that the $\W_1$ can be closely approximated by a shifted gamma distribution, so that 
\begin{eqnarray}
\label{eq:Fbetaappx}
F_{1}(x) \triangleq \Pr(\W_1 \leq x) &\simeq& 
   \gammareg\left(\kappx,\frac{x+\aappx}{\tappx}\right) \qquad\qquad\qquad x>-\aappx  \\ 
x = F^{-1}_{1}(y) &\simeq& \tappx \, \gammareg^{-1}(\kappx,y)-\aappx \label{eq:F1invbetaappx}
\end{eqnarray}
where  $\gammareg(a,x)$ is the regularized lower incomplete gamma function, $\gammareg^{-1}(a,y)$ is its  inverse, and the constants $\kappx=46.446, \tappx=0.186054, \aappx=9.84801$ have been chosen to match the moments of the approximation to that of the Tracy-Widom. 

Thus, 
 using \eqref{eq:Fbetaappx} in \eqref{eq:logit} we obtain for the \ac{CDF} of the Roy's statistic in the null case
\begin{equation}\label{eq:FX1appx}
F_{\X_1}(\x_1) \simeq  \gammareg\left(\kappx, \frac{{\log\left(\x_1/(1-\x_1)\right)-\mu} +{\sigma} \aappx}{ \tappx \sigma}\right) 
\end{equation}
 and for its inverse, useful for evaluating the percentiles,
\begin{equation}\label{eq:FX1invappx}
\x_1 = F^{-1}_{\X_1}(y) \simeq  \frac{\exp\left\{\sigma [\tappx \gammareg^{-1}(\kappx,y)-\aappx]+\mu \right\}}{1+\exp\left\{\sigma [\tappx \gammareg^{-1}(\kappx,y)-\aappx]+\mu \right\}} \,.
\end{equation}
%
An example comparing exact (by Algorithm 1) with approximate (by \eqref{eq:FX1appx}) distributions is reported in Fig. \ref{fig:cdfbeta}.
\begin{figure}[h]
\centerline{\includegraphics[width=0.9\columnwidth,draft=false] {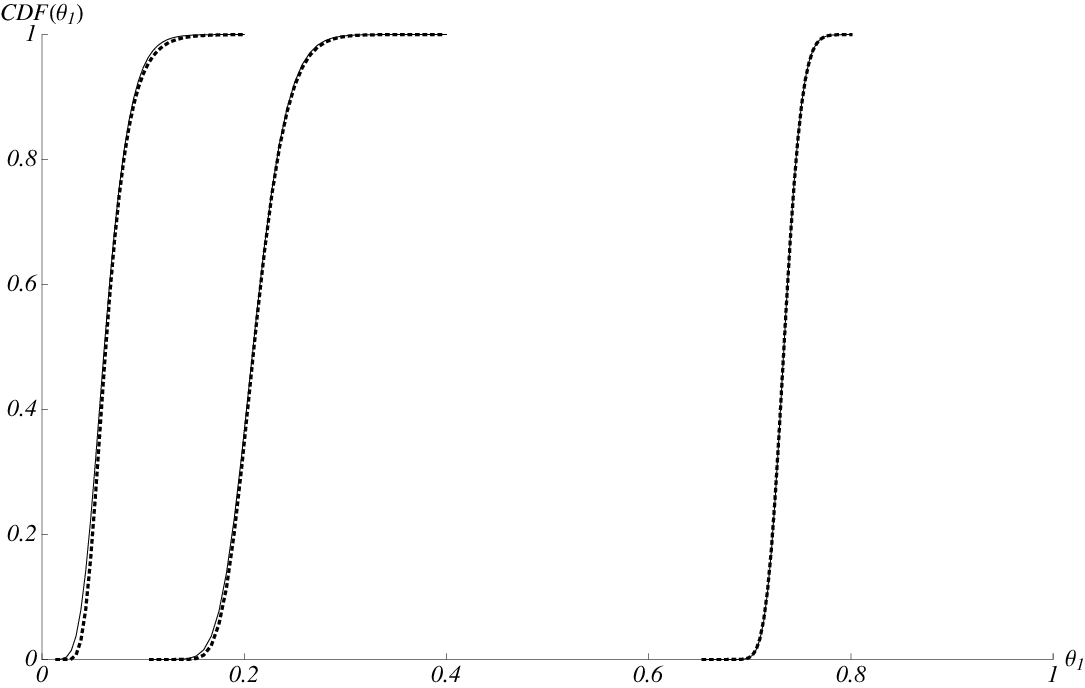}}
\caption{\ac{CDF} of the largest eigenvalue, real beta matrix, $m=-1/2, n=100$. From left to right: $\nmin=5, 15, 100$. Comparison between the exact distribution \eqref{eq:cdfwishartreal} (solid line) and the approximation in \eqref{eq:FX1appx} (dotted). Note that for $\nmin=100$ the two curves are almost indistinguishable.} \label{fig:cdfbeta}
\end{figure}

While Theorem 1 with Algorithm 1 gives the exact distribution, the two approximations above, which are in simple closed forms, can be used for a rapid, 
 approximate design of the Roy's test.  For instance, with $s=5, m=-1/2, n=1000$ the 80th-percentile obtained by root-finding with the exact CDF of Algorithm 1 is $\x_1=0.008501$, while \eqref{eq:FX1invappx} gives $\x_1=0.008609$. For larger matrices and larger values of the \ac{CDF} the approximation generally improves. For example, with 
 $s=200, m=-1/2, n=299/2$ the 99th-percentile obtained by Algorithm 1 is $\x_1=0.827760$, while \eqref{eq:FX1invappx} gives $\x_1=0.827761$.
%


 \section{Remarks for the complex multivariate beta}
 
 For completeness we recall that, when ${\bf X, Y} $ are two independent {\emph{complex}} Gaussian, the analogous of \eqref{eq:jpdfuncorrnodet} is the complex multivariate beta, where the joint distribution of the eigenvalues is \cite{Kha:64}
\begin{equation}\label{eq:jpdfuncorrnodetcomplex}
f(\x_{1}, \ldots, \x_{\Nmin}) = C'(s,m,n) \,
     \prod_{i=1}^{\Nmin} \x_{i}^{m} (1-\x_{i})^n
    \cdot    \prod_{i<j}^{\Nmin}
    \left(\x_{i}-\x_{j}\right)^2 
\end{equation}
where $1 > \x_{1} \geq \x_{2} \cdots \geq \x_{\Nmin} > 0$, and $$C'(s,m,n)=\prod _{i=1}^s  \frac{ \Gamma \left({m+ n+s+i}\right)}{ \Gamma
   \left({i}\right) \Gamma \left({i+m} \right) \Gamma \left({i+n} \right)} \,.$$ 
In this case we can write 
\begin{equation} \label{eq:jpdfuncorrcomplex}
f(\x_{1}, \ldots, \x_{\Nmin}) = C'(s,m,n) \left|{\bf V}({\bf \x})\right|^2  \prod_{i=1}^{\Nmin} \xi(\x_{i})  \,.
\end{equation}
Therefore, by applying \cite[Corollary 2]{ChiWinZan:J03} the \ac{CDF} of $\X_1$ in the complex case is simply given by \cite{Kha:64}
\begin{equation}
\label{eq:cdfwishartcomplex}
F_{\X_1}(\x_1)=\Pr\left\{\X_1 \leq \x_1\right\}=C'(s,m,n) \,   \left|{\bf A}(\x_1)\right| \,
\end{equation}
%
where the elements of the $\nmin \times \nmin$ matrix ${\bf A}(\x_1)$  are:
\begin{equation}
\label{eq:aijcomplex}
a_{i,j}(\x_1)=
\betareg(\x_1, m+i+j-1,n+1) \qquad i,j=1,\ldots,\nmin \,.
\end{equation}
Moreover, expressions analogous to \eqref{eq:Fbetaappx} and \eqref{eq:FX1invappx} can be derived, by using the 
 Tracy-Widom limiting behavior for  $\X_1$ in the complex case \cite{Joh:08} and the approximation of the Tracy-Widom of order 2 with a shifted gamma \cite{Chi:J14}. 

\appendix

\bibliographystyle{elsarticle-num-names-alpha}
\bibliography{BiblioMCCV,RandomMatrix,MyBooks}

\end{document}